\documentclass[twoside,a4paper,reqno,11pt]{amsart}
\usepackage[top=28mm,right=30mm,bottom=28mm,left=30mm]{geometry}

\usepackage{amsfonts, amsmath, amssymb, mathrsfs, bm, latexsym, stmaryrd, array, hyperref, mathtools, bold-extra}

\renewcommand{\a}{\alpha}
\renewcommand{\b}{\beta}

\newcommand{\e}{\varepsilon}
\renewcommand{\l}{\lambda}

\renewcommand{\O}{\Omega}

\newcommand{\normeq}{\trianglelefteqslant}

\newcommand{\Sym}{\mathrm{Sym}}

\newcommand{\la}{\langle}
\newcommand{\ra}{\rangle}

\renewcommand{\to}{\rightarrow}

\newcommand{\leqs}{\leqslant}
\newcommand{\geqs}{\geqslant}

\newcommand{\vs}{\vspace{2mm}}

\makeatletter
\newcommand{\imod}[1]{\allowbreak\mkern4mu({\operator@font mod}\,\,#1)}
\makeatother

\newtheorem{theorem}{Theorem}
\newtheorem*{conj*}{Conjecture}

\newtheorem{corol}[theorem]{Corollary}

\newtheorem{thm}{Theorem}[section]
\newtheorem{prop}[thm]{Proposition}
\newtheorem{lem}[thm]{Lemma}

\theoremstyle{definition}
\newtheorem{rem}[thm]{Remark}

\begin{document}

\author{Timothy C. Burness}
\address{T.C. Burness, School of Mathematics, University of Bristol, Bristol BS8 1UG, UK}
\email{t.burness@bristol.ac.uk}

  \author{Aner Shalev}
  \thanks{AS acknowledges the support of ISF grant 686/17 and the Vinik Chair of Mathematics which he holds. Both authors thank Joanna Fawcett and Michael Giudici for helpful comments on twisted wreath products. They also thank an anonymous referee for many insightful comments, corrections and suggestions on an earlier version of this paper.}
  \address{A. Shalev, Institute of Mathematics, Hebrew University, Jerusalem 91904, Israel}
\email{shalev@math.huji.ac.il}

\title[Permutation groups with restricted stabilizers]{Permutation groups with restricted stabilizers}
\dedicatory{Dedicated to the memory of our friend and colleague Jan Saxl}

\begin{abstract}
Fix a positive integer $d$ and let $\Gamma_d$ be the class of finite groups without sections isomorphic to the alternating group $A_d$. The groups in $\Gamma_d$ were studied by Babai, Cameron and P\'{a}lfy in the 1980s and they determined bounds on the order of a primitive permutation group with this property, which have found a wide range of applications. Subsequently, results on the base sizes of such groups were also obtained. In this paper we replace the structural conditions on the group by restrictions on its point stabilizers, and we obtain similar, and sometimes stronger conclusions. For example, we prove that there is a linear function $f$ such that the base size of any finite primitive group with point stabilizers in $\Gamma_d$ is at most $f(d)$. This generalizes a recent result of the first author on primitive groups with solvable point stabilizers. For non-affine primitive groups we obtain stronger results, assuming only that stabilizers of $c$ points lie in $\Gamma_d$.  We also show that if $G$ is any permutation group of degree $n$ whose $c$-point stabilizers lie in $\Gamma_d$, then $|G| \leqs ((1+o_c(1))d/e)^{n-1}$. This asymptotically extends and improves a $d^{n-1}$ upper bound on $|G|$ obtained by Babai, Cameron and P\'{a}lfy assuming $G \in \Gamma_d$.
\end{abstract}

\date{\today}

\maketitle
 \section{Introduction}\label{s:intro}

For $d \geqs 5$, let $\Gamma_d$ be the class of finite groups that have no section isomorphic to the alternating group $A_{d}$. A celebrated result of Babai, Cameron and P\' {a}lfy \cite{BCP} shows that there exists a function $f_1$ with the property that $|G| \leqs n^{f_1(d)}$ for every primitive permutation group $G \in \Gamma_d$
of degree $n$. This result has been useful in a wide range of applications. For example, the relevant condition on sections arises naturally in the study of subgroup growth and bounded generation of profinite groups (see \cite{Py,Sh97} for instance) and we refer the reader to \cite{Bab15, PPSS} for applications in algebraic graph theory.

Recall that the \emph{base size} $b(G)$ of a permutation group $G \leqs {\rm Sym}(\O)$ is the minimal cardinality of a subset of $\O$ whose pointwise stabilizer is trivial. Clearly, if $G$ has degree $n$, then $|G| \leqs n^{b(G)}$. By combining this observation with the above bound of Babai et al., it is natural to ask if there is a suitable function $f_2$ such that $b(G) \leqs f_2(d)$ for every primitive group $G \in \Gamma_d$. The existence of such a function was conjectured by Babai and proved by Gluck, Seress and Shalev \cite{GSSh}, with $f_2$ a quadratic function of $d$. This was subsequently improved by Liebeck and Shalev \cite{LSh} who showed that the result holds with a linear function $f_2$.

In this paper we seek analogous results for finite permutation groups, where we replace the structural conditions on the whole group by similar restrictions on its point stabilizers, thus investigating a larger class of groups. Our first result deals with the base size of primitive groups $G$ whose point stabilizers lie in $\Gamma_d$. In our second result we impose the same restriction on smaller subgroups, namely on point stabilizers of $c$ distinct points, and we do not assume that $G$ is primitive, or even transitive.

\begin{theorem}\label{t:main}
Let $d \geqs 5$ be an integer. Then there exists a linear function $f(d)$ such that $b(G) \leqs f(d)$ for every finite primitive permutation group $G$ with point stabilizers in $\Gamma_d$.
\end{theorem}

Notice that a linear upper bound is best possible. For example, if we view $G= {\rm PGL}_{m}(q)$ as a primitive permutation group on the set of $1$-dimensional subspaces of the natural module, then $b(G) = m+1$ and each point stabilizer is contained in $\Gamma_{cm}$ for some absolute constant $c$.

Next, we bound the order of an arbitrary permutation group $G$ on a set $\O$ of size $n$ whose point stabilizers lie in $\Gamma_d$.
In fact we prove a more general result, fixing any nonnegative integer $c$ and assuming that all pointwise stabilizers of subsets of size $c$ lie in $\Gamma_d$. Recall that if $\Delta$ is a subset of $\O$, then $G_{(\Delta)} = \bigcap_{\a \in \Delta}G_{\a}$ and $G_{\Delta}$ denote the pointwise and setwise stabilizer of $\Delta$, respectively. Note that if $\Delta$ is empty then we define $G_{(\Delta)} = G_{\Delta} = G$.

\begin{theorem}\label{t:3}
Let $c \geqs 0$ be an integer and let $\delta > 0$ be a real number. Then there exists $N(c,\delta) \in \mathbb{N}$ such that if $d \geqs N(c,\delta)$ and $G \leqs {\rm Sym}(\O)$ has degree $n$ with $G_{(\Delta)} \in \Gamma_d$ for every subset $\Delta \subseteq \O$ of size $c$, then $|G| \leqs ((1+\delta)d/e)^{n-1}$.
\end{theorem}

It is easy to see that the number $N(c,\delta)$ in Theorem \ref{t:3} must genuinely depend on both $c$ and $\delta$. For example, let $n=c=2d$ and $G=S_n$. Then $|\Delta| = c$ implies $\Delta = \O$ and $G_{(\Delta)} = 1 \in \Gamma_d$.
But $|G| = n! = (2d)! > ((1+\delta)d/e)^{n-1}$ for all $\delta < 1$. In general, $N(c,\delta) \to \infty$ as $c \to \infty$ or as $\delta \to 1$. It is also worth noting that our proof of Theorem \ref{t:3} is independent of the Classification of Finite Simple Groups (see Remark \ref{r:cfsg}).

As an immediate corollary of Theorem \ref{t:3}, we get the following result. Indeed, parts (i) and (ii) are Theorem \ref{t:3} with $c=0$ and $c=1$, respectively. Part (iii) also follows from Theorem \ref{t:3} since $G_{(\Delta)} \leqs G_{\Delta}$.

\begin{corol}\label{t:2}
Let $\delta>0$ be a real number.
\begin{itemize}\addtolength{\itemsep}{0.2\baselineskip}
\item[{\rm (i)}] There exists $N_1(\delta) \in \mathbb{N}$ such that if $d \geqs N_1(\delta)$ and $G \in \Gamma_d$ is a permutation group of degree $n$, then $|G| \leqs ((1+\delta)d/e)^{n-1}$.

\item[{\rm (ii)}] There exists $N_2(\delta) \in \mathbb{N}$ such that if $d \geqs N_2(\delta)$ and $G \leqs {\rm Sym}(\O)$ has degree $n$ with $G_{\a} \in \Gamma_d$ for all $\a \in \O$, then $|G| \leqs ((1+\delta)d/e)^{n-1}$.

\item[{\rm (iii)}] For every integer $c \geqs 0$, there exists $N(c,\delta) \in \mathbb{N}$ such that if $d \geqs N(c,\delta)$ and $G \leqs {\rm Sym}(\O)$ has degree $n$ with $G_{\Delta} \in \Gamma_d$ for every subset $\Delta \subseteq \O$ of size $c$, then $|G| \leqs ((1+\delta)d/e)^{n-1}$.
\end{itemize}
\end{corol}

Theorem \ref{t:3} and Corollary \ref{t:2} asymptotically extend and improve a useful bound of Babai, Cameron and P\' {a}lfy \cite[Lemma 2.2]{BCP}, which states that if $d \geqs 6$ then $|G| \leqs d^{n-1}$ for every permutation group $G \in \Gamma_d$ of degree $n$ (see also Lemma \ref{l:2} below).
Indeed, our assumption on $G$ in Theorem \ref{t:3} is weaker, while the conclusion is asymptotically stronger. Note that the upper bounds on $|G|$
in Theorem \ref{t:3} and Corollary \ref{t:2} are essentially best possible, as shown by the example $n=d$ and $G = S_d$.

\vs

A well known theorem of Seress \cite{Seress} states that $b(G) \leqs 4$ for every  finite primitive solvable group $G$ and this bound is best possible. This has very recently been extended by Burness \cite{Bur20}, who has proved that $b(G) \leqs 5$ for every finite primitive group with solvable point stabilizers (this bound is also optimal). Therefore, one can view Theorem \ref{t:main} as a natural generalization of the main result of \cite{Bur20}.

Our next result can be viewed as a further extension of \cite{Bur20} by imposing a condition on $2$-point stabilizers.

\begin{theorem}\label{t:4}
Let $G$ be a finite primitive non-affine permutation group with point stabilizer $H$ and assume every $2$-point stabilizer in $G$ is solvable. Then $b(G) \leqs 6$, and this bound is best possible. Moreover, if $G$ is almost simple, then $b(G) = 6$ if and only if $G = {\rm Sp}_{6}(2)$ and
$H = {\rm O}_{6}^{+}(2)$.
\end{theorem}

We refer the reader to Remark \ref{r:fn} for comments on the affine groups excluded in Theorem \ref{t:4} (we are not aware of any exceptions to the bound $b(G) \leqs 6$).

Finally, we conclude by establishing the following extension of Theorem \ref{t:main} to $c$-point stabilizers in non-affine primitive groups.

\begin{theorem}\label{t:5}
Let $c,d$ be positive integers with $d \geqs 5$. Then there exists a function $g(c,d)$, linear in $c$ and $d$, such that $b(G) \leqs g(c,d)$ for every finite primitive non-affine permutation group $G$ with $G_{(\Delta)} \in \Gamma_d$ for all subsets $\Delta \subseteq \O$ of size $c$.
\end{theorem}

We plan to study related problems, including the affine case of Theorems \ref{t:4} and \ref{t:5} in a subsequent paper.

\vs

In Section \ref{s:proof} we present proofs of Theorems \ref{t:main} and \ref{t:3}. The proofs of Theorems \ref{t:4} and \ref{t:5} are given in Sections \ref{s:proof2} and \ref{s:fn}, respectively. Our notation is standard. In particular, we adopt the notation from \cite{KL} for simple groups and we note that all logarithms in this paper are base two.

\section{Proof of Theorems \ref{t:main} and \ref{t:3}}\label{s:proof}

In this section, we prove Theorems \ref{t:main} and \ref{t:3}. We begin by recording some preliminary results.

\subsection{Preliminaries}\label{ss:prel}

\begin{lem}\label{l:1}
Let $d \geqs 5$ be an integer. The class $\Gamma_d$ is closed under subgroups, quotients and extensions.
\end{lem}

\begin{proof}
The fact that $\Gamma_d$ is closed under subgroups and quotients is immediate from the definition. Now assume $G$ is a finite group with a normal subgroup $N$ such that both $N$ and $G/N$ are in $\Gamma_d$. Seeking a contradiction, suppose $L/K \cong A_d$ is a section of $G$. Consider the natural map $\varphi:L/K \to LN/KN$ and note that $\varphi$ is either trivial or injective. If $\varphi$ is trivial, then $L  = K(L\cap N)$ and $L/K \cong (L \cap N) / (K \cap N)$, which is a contradiction since $N \in \Gamma_d$. Similarly, if $\varphi$ is injective then $L/K \cong LN/KN$, which contradicts the fact that $G/N \in \Gamma_d$. The result follows.
\end{proof}

\begin{lem}\label{l:2}
Let $G$ be a permutation group of degree $n$ such that $G \in \Gamma_d$ with $d \geqs 2$. Then $|G| < d^{n-1}$.
\end{lem}

\begin{proof}
This is essentially \cite[Lemma 2.2]{BCP}, except that we assume $d \geqs 2$ instead of $d \geqs 6$. In addition, it is worth noting that our proof uses a result from \cite{Maroti}, which relies on the Classification of Finite Simple Groups, whereas the proof of \cite[Lemma 2.2]{BCP} is CFSG-free.

First note that $\Gamma_2$ is empty, so there is nothing to prove for $d=2$. Now assume $d \geqs 3$. If $n < d$ then $|G| \leqs n! \leqs n^{n-1} < d^{n-1}$, so we may assume $n \geqs d$. This implies that $G \ne A_n, S_n$.

We prove the upper bound on $|G|$ by induction on $n$.
If $G$ has an orbit of length $k<n$ then the induction hypothesis gives $|G| < d^{k-1}\cdot d^{n-k-1} < d^{n-1}$. Similarly, if $G$ is transitive with blocks of imprimitivity of size $1<k<n$, then $|G| < (d^{k-1})^{n/k}\cdot d^{n/k-1} = d^{n-1}$.
It remains to deal with the case where $G$ is primitive.

Suppose $d=3$, so $|G|$ is indivisible by $3$. Since $G \ne A_n, S_n$,
\cite[Corollary 1.4]{Maroti} implies that either $|G| \leqs 2^{n-1}$, or $n=5$ and $G = {\rm AGL}_{1}(5)$. In both cases we have $|G| < 3^{n-1}$, as required. Finally, suppose $d \geqs 4$, so $n \geqs 4$.  If $n=4$ then $|G| \leqs |S_4| < d^3$. For $n > 4$, \cite[Corollary 1.2]{Maroti} gives $|G| < 3^n \leqs 4^{n-1} \leqs d^{n-1}$ and the proof is complete.
\end{proof}

\begin{lem}\label{l:3}
Suppose $G$ is a finite simple classical group with natural module $V$. If $G \in \Gamma_d$ with $d \geqs 5$, then $\dim V < kd$, where $k = 1$ if $G$ is linear or unitary, otherwise $k=2$.
\end{lem}

\begin{proof}
First assume $G = {\rm L}_{m}(q)$ and let $H$ be the stabilizer in $G$ of a direct sum decomposition of $V$ into $1$-spaces. Then  $A_m$ is a section of $H$ (see \cite[Proposition 4.2.9]{KL}, for example) and thus $m<d$ as required. A very similar argument applies if $G = {\rm U}_{m}(q)$ or $\O_m(q)$ (with $mq$ odd), working with an orthogonal decomposition of $V$ into nondegenerate $1$-spaces. Finally, if $G = {\rm PSp}_{m}(q)$ or ${\rm P\O}_{m}^{\pm}(q)$ then $m$ is even and by considering the stabilizer of an appropriate orthogonal decomposition of $V$ into nondegenerate $2$-spaces we deduce that $A_{m/2}$ is a section of $G$. This gives $m/2<d$ and the result follows.
\end{proof}

Let $R \leqs {\rm Sym}(\Delta)$ be a permutation group on a finite set $\Delta$ and recall that the \emph{distinguishing number} of $R$, denoted $d(R)$, is the minimal number of colours needed to colour the points in $\Delta$ in such a way that no nontrivial element of $R$ preserves the colouring. We will need the following theorem to handle product type groups.

\begin{thm}\label{t:dist}
Let $R \leqs {\rm Sym}(\Delta)$ be a permutation group on a finite set $\Delta$.
\begin{itemize}\addtolength{\itemsep}{0.2\baselineskip}
\item[{\rm (i)}] If $R$ is solvable, then $d(R) \leqs 5$.
\item[{\rm (ii)}] If $R \in \Gamma_d$ for some positive integer $d$, then $d(R) \leqs d$.
\end{itemize}
\end{thm}

\begin{proof}
Part (i) is \cite[Theorem 1.2]{Seress} and part (ii) is \cite[Theorem 2.3]{HP}.
\end{proof}

\begin{thm}\label{t:gss}
There exists a linear function $g(d)$ such that $b(G) \leqs g(d)$ for every finite primitive permutation group $G$ in $\Gamma_d$.
\end{thm}

\begin{proof}
The existence of a quadratic function $g(d)$ with this property is the main theorem of \cite{GSSh}.
In later work, the existence of a linear function was established in \cite[Theorem 1.4]{LSh}.
\end{proof}

\subsection{Proof of Theorem \ref{t:main}} We are now ready to prove Theorem \ref{t:main}. Fix a positive integer $d \geqs 5$ and let $G \leqs {\rm Sym}(\O)$ be a finite primitive permutation group of degree $n$ with point stabilizer $H \in \Gamma_d$. We proceed by considering the possibilities for $G$ given by the Aschbacher-O'Nan-Scott Theorem (see Table \ref{tab:prim}).

\begin{table}
\[
\begin{array}{ll} \hline
\mbox{Type} & \mbox{Description} \\ \hline
\mbox{I} & \mbox{Affine: $G = V{:}H \leqs {\rm AGL}(V)$, $H \leqs{\rm GL}(V)$ irreducible} \\
\mbox{II} & \mbox{Almost simple: $T \leqs G \leqs {\rm Aut}(T)$} \\
\mbox{III(a)(i)} & \mbox{Diagonal type: $T^k \leqs G \leqs T^k.({\rm Out}(T) \times P)$, $P \leqs S_k$ primitive} \\
\mbox{III(a)(ii)} & \mbox{Diagonal type: $T^2 \leqs G \leqs T^2.{\rm Out}(T)$} \\
\mbox{III(b)(i)} & \mbox{Product type: $G \leqs L \wr P$, $L$ primitive of type II, $P \leqs S_k$ transitive} \\
\mbox{III(b)(ii)} & \mbox{Product type: $G \leqs L \wr P$, $L$ primitive of type III(a), $P \leqs S_k$ transitive} \\
\mbox{III(c)} & \mbox{Twisted wreath product} \\ \hline
\end{array}
\]
\caption{The finite primitive permutation groups}
\label{tab:prim}
\end{table}

\begin{prop}\label{p:1}
Suppose $G$ is either an affine or diagonal type group. Then $G \in \Gamma_d$ and the conclusion to Theorem \ref{t:main} holds.
\end{prop}

\begin{proof}
In view of Theorem \ref{t:gss}, it suffices to show that $G \in \Gamma_d$. Let $N=T^k$ denote the socle of $G$, where $T$ is a simple group. If $G$ is affine then $G$ is an extension of $N$ by $H$ and the result follows from Lemma \ref{l:1} since $N$ is abelian. Now assume $G$ is a diagonal type group, so  $T$ is nonabelian and $k \geqs 2$. Here $N \normeq G \leqs N.({\rm Out}(T) \times S_k)$, $n = |T|^{k-1}$ and $H \cong T.(G/N)$. Since $H \in \Gamma_d$, it follows that $T \in \Gamma_d$ and $G/N \in \Gamma_d$. Therefore, both $N$ and $G/N$ are in $\Gamma_d$ and we conclude that $G \in \Gamma_d$ as required.
\end{proof}

\begin{prop}\label{p:tw}
The conclusion to Theorem \ref{t:main} holds if $G$ is a twisted wreath product.
\end{prop}

\begin{proof}
Let $N=T^k$ be the socle of $G$, where $T$ is a nonabelian simple group. Then $G = NH$ is a split extension, $N$ is regular and $H \leqs S_k$ is transitive. Therefore, $n=|T|^k$ and as explained in \cite[Section 4.4]{HLM}, we have
\[
b(G) \leqs 2\frac{\log |G|}{\log n} + 24 = \frac{2\log |H|}{k\log |T|}+26.
\]
Now Lemma \ref{l:2} gives $|H| \leqs d^{k-1}$ and thus
\[
b(G) \leqs \frac{2(k-1)\log d}{k \log |T|}+26 < 2\frac{\log d}{\log |T|}+26 \leqs 2\frac{\log d}{\log 60}+26.
\]
The result follows.
\end{proof}

\begin{rem}\label{r:tw}
In the proof of Proposition \ref{p:tw}, we work with the bound
\[
b(G) \leqs 2\frac{\log |G|}{\log n} + 24
\]
from \cite[Section 4.4]{HLM}. We thank Joanna Fawcett (personal communication) for pointing out to us that the stronger bound
\begin{equation}\label{e:faw}
b(G) \leqs \left\lceil \frac{\log |G|}{\log n} \right\rceil+ 3 < \frac{\log |G|}{\log n} + 4
\end{equation}
holds, which will be important in the proof of Proposition \ref{p:tw2}. To see this, we simply combine the bounds
\[
b(G) \leqs \left\lceil \frac{\log d(H)}{\log |T|}\right\rceil+3
\]
and $d(H) \leqs 48\sqrt[k]{|H|}$ from \cite[Theorem 1.3]{fawcett_tw} and \cite[Theorem 1.2]{DHM}, with the observation that $|T| \geqs 60$.
\end{rem}

\begin{prop}\label{p:as}
The conclusion to Theorem \ref{t:main} holds if $G$ is almost simple.
\end{prop}

\begin{proof}
Let $G_0$ denote the socle of $G$, so $G_0 \normeq G \leqs {\rm Aut}(G_0)$ and $G_0$ is a nonabelian simple group. Recall that $G$ is \emph{standard} if $G_0 = A_m$ and $\O$ is a set of subsets or partitions of $\{1, \ldots, m\}$, or $G_0$ is a classical group acting on an orbit of subspaces (or pairs of subspaces of complementary dimension) of the natural module $V$ for $G_0$. In all other cases, $G$ is said to be \emph{nonstandard}. By the main theorem of \cite{BLSh} we have $b(G) \leqs 7$ for all nonstandard groups, so for the remainder we may assume $G$ is standard.

First assume $G_0 = A_m$ is an alternating group with $m \geqs 7$ and $G$ is standard, so either
\begin{itemize}\addtolength{\itemsep}{0.2\baselineskip}
\item[{\rm (a)}] $H = (S_k \times S_{m-k}) \cap G$ is intransitive, where $1 \leqs k < m/2$; or
\item[{\rm (b)}] $H = (S_k \wr S_{m/k}) \cap G$ is imprimitive, where $1< k \leqs m/2$ and $k$ divides $m$.
\end{itemize}
In (a), a result of Halasi \cite[Corollary 4.3]{Halasi} gives
\[
b(G) \leqs \left\lceil \log_{\lceil m/k \rceil}m\right\rceil\left(\left\lceil m/k \right\rceil - 1\right) \leqs m-1
\]
and the result follows since the hypothesis $H \in \Gamma_d$ implies that $m \leqs 2d$. Now consider case (b). Here $k$ and $m/k$ are both at most $d$ and \cite[Theorem 4]{BCN} gives
\[
b(G) \leqs \max\left\{6, \left\lceil \log_{m/k} k \right\rceil+3\right\} \leqs \log d + 4
\]
for $k \geqs 3$. On the other hand, if $k=2$ and $m \geqs 8$ then $b(G) \leqs 3$ (see \cite[Remark 1.6]{BGS}). In both cases, the desired result follows.

To complete the proof of the proposition, we may assume $G_0$ is a finite simple classical group over $\mathbb{F}_q$ with natural module $V$ of dimension $m$. Since we are assuming $G$ is standard, either
\begin{itemize}\addtolength{\itemsep}{0.2\baselineskip}
\item[{\rm (a)}] $H$ is the stabilizer in $G$ of a subspace (or pair of subspaces) of $V$; or
\item[{\rm (b)}] $G_0 = {\rm Sp}_{m}(q)$, $q$ is even and $H \cap G_0 = {\rm O}_{m}^{\pm}(q)$.
\end{itemize}
First consider the cases arising in (a). Here the structure of $H$ is given in \cite[Section 4.1]{KL} and by applying Lemma \ref{l:3} we deduce that $m<4d$. By combining Theorem 3.3 and Proposition 3.5 in \cite{HLM}, and appealing to the proof of \cite[Theorem 3.1]{HLM}, we deduce that $b(G) \leqs m+14< 4d+14$ and the result follows. Similarly, in (b) we have $b(G) \leqs m+4 < 2d+4$ (see \cite[Section 3.3]{HLM}).
\end{proof}

In order to complete the proof of Theorem \ref{t:main}, we may assume that $G$ is a primitive group of product type.

\begin{prop}\label{p:prod}
The conclusion to Theorem \ref{t:main} holds if $G$ is a product type group.
\end{prop}

\begin{proof}
Let $N = T^k$ be the socle of $G$ and write $N \leqs G \leqs L \wr P$, where $L \leqs {\rm Sym}(\Delta)$ is a primitive group with socle $T$ (either almost simple or diagonal type), $P$ is the transitive subgroup of $S_k$ induced by the conjugation action of $G$ on the $k$ factors of $N$ and
\[
\O = \Delta_1 \times \cdots \times \Delta_k = \Delta^k.
\]
We will denote an arbitrary element $x$ in $G$ by writing $x=(x_1, \ldots, x_k)\pi$ with $x_i \in L$ and $\pi \in P$.

Fix $\delta \in \Delta$ and let $\a = (\delta, \ldots, \delta) \in \O$, so $G_{\a} = \{(x_1, \ldots, x_k)\pi \in G \,:\, \mbox{$x_i \in L_{\delta}$ for all $i$}\}$. Set
\[
G_1 = \{(x_1, \ldots, x_k)\pi \in G \,:\, 1^{\pi}=1\}.
\]
By \cite[2.2]{Kov}, we may assume that $G_1$ induces $L$ on $\Delta_1$. As a consequence, we claim that $(G_1)_{\a}$ induces $L_{\delta}$ on $\Delta_1$.
To see this, suppose $z_1 \in L_{\delta}$ and write $z = (z_1, \ldots, z_k)\pi \in G$ with $1^{\pi}=1$. Since $T = {\rm soc}(L)$ and $L$ is primitive, it follows that $T$ acts transitively on $\Delta$ and so there exist $t_i \in T$ such that $t_iz_i \in L_{\delta}$ for $i=2, \ldots, k$. Now $G$ contains $t = (1,t_2, \ldots, t_k) \in N$, so $tz = (z_1, t_2z_2, \ldots, t_kz_k)\pi \in (G_1)_{\a}$ and this justifies the claim. It follows that every point stabilizer $L_{\delta}$ is in $\Gamma_d$ and thus $b(L) \leqs g(d)$ for some linear function of $d$ by Propositions \ref{p:1} and \ref{p:as}.

Let $\{\delta_1, \ldots, \delta_b\} \subseteq \Delta$ be a base for $L$ with $b=b(L)$ and set $\a_i = (\delta_i, \ldots, \delta_i) \in \O$ for $i=1, \ldots, b$. Then
\[
Q := \bigcap_{i=1}^{b}G_{\a_i} = \{ (x_1, \ldots, x_k)\pi  \in G \,:\, \mbox{$x_j=1$ for all $j$} \}
\]
and we may view $Q$ as a subgroup of $P$. By combining \cite[Lemma 3.8]{BS} with \cite[Lemma 2.1]{DHM} we deduce that
\begin{equation}\label{e:prod}
b(G) \leqs \left\lceil \log_{|\Delta|}d(Q)\right\rceil+b(L),
\end{equation}
where $d(Q)$ is the distinguishing number of $Q$ in its natural action on $\{1, \ldots, k\}$.
Since $N$ is transitive, we have $G = NH$ and thus $H$ induces $P$ on the set of factors of $T^k$. In particular, $P$ is contained in $\Gamma_d$, so $Q \in \Gamma_d$ and thus Theorem \ref{t:dist}(ii) gives $d(Q) \leqs d$. The result now follows from \eqref{e:prod}.
\end{proof}

This completes the proof of Theorem \ref{t:main}.

\subsection{Proof of Theorem \ref{t:3}}

We begin by establishing the following result, which improves the bound on $|G|$ in Lemma \ref{l:2} for large $d$. Notice that the statement coincides with part (i) of Corollary \ref{t:2}.

\begin{prop}\label{p:new}
Let $\e > 0$ be a real number. Then there exists $M(\e) \in \mathbb{N}$ such that if $d \geqs M(\e)$, $G \leqs {\rm Sym}(\O)$ has degree $n$ and $G \in \Gamma_d$,
then $|G| \leqs ((1+\e)d/e)^{n-1}$.
\end{prop}

\begin{proof}
Given $\e > 0$, let $M(\e)$ be the minimal integer $M \geqs 5e$
such that for all $m \geqs M$ we have
\[
m^{3/2} \leqs (1 + \e)^{m/e-1}.
\]
Let $d \geqs M(\e)$. We prove the result by induction on $n$, noting that the case $n=1$ is trivial. Let us assume $n \geqs 2$.

We apply the known bound $n! \leqs e^{1/12n}(2 \pi n)^{1/2} (n/e)^n$ which holds for all $n \geqs 1$
(see for instance \cite{F}, Section 2.9). Since $n \geqs 2$ and $e^{1/24} (2 \pi)^{1/2} < e$ it follows that
\[
n! < e n^{1/2} (n/e)^n.
\]

Suppose first that $n \leqs d-1$. If $n \leqs d/e$ then $n! \leqs n^{n-1} \leqs (d/e)^{n-1} < ((1+\e)d/e)^{n-1}$,
as required. Otherwise we have $d/e < n \leqs d-1$. This yields
\[
|G| \leqs n! <  e n^{1/2} (n/e)^n < e d^{1/2} (d/e)^n \leqs ((1+\e)d/e)^{n-1},
\]
where the last inequality follows from the fact that $d \geqs M(\e)$ so that $d^{3/2} \leqs (1 + \e)^{d/e-1}$.

It remains to prove the result for $n \geqs d$.
The argument is now similar to the proof of Lemma \ref{l:2}. Set $d_1 := (1+\e)d/e$.

If $G$ has an orbit of length $k<n$ then induction yields $|G| \leqs d_1^{k-1}\cdot d_1^{n-k-1} < d_1^{n-1}$.
If $G$ is transitive with blocks of imprimitivity of size $1<k<n$, then by the induction hypothesis we obtain
$|G| \leqs (d_1^{k-1})^{n/k}\cdot d_1^{n/k-1} = d_1^{n-1}$.
Finally, if $G$ is primitive, then $G \in \Gamma_d$ implies that $A_n \not\leqs G$ and the main theorem of \cite{PS} gives $|G|<4^{n}$. Since $n \geqs d \geqs 5e$, it follows that
$|G| < 5^{n-1} < d_1^{n-1}$, completing the proof.
\end{proof}

\begin{rem}\label{r:cfsg}
It is worth noting that the proof of Proposition \ref{p:new} is independent of Lemma \ref{l:2}. In particular, in the final step we use \cite{PS} to bound the order of $G$, which does not rely on the Classification of Finite Simple Groups. This is in contrast to the proof of Lemma \ref{l:2}, where we applied a theorem of Mar\'{o}ti \cite{Maroti} that does require CFSG. This small adjustment in the proof of Proposition \ref{p:new} allows us to present a CFSG-free proof of Theorem \ref{t:3}.
\end{rem}

\begin{proof}[Proof of Theorem \ref{t:3}]
Let $\delta > 0$ be a real number and let $c \geqs 0$ be an integer. We proceed by induction on $c$, noting that the case $c=0$ is Proposition \ref{p:new}, where we define $N(0,\delta) := M(\delta)$.

Now assume $c \geqs 1$. Set
\[
\e = (1+\delta)^{3/5}-1,
\]
so $0 < \e < \delta$, and define
\[
N(c,\delta) := \max(N(c-1, \e), M(\e), c),
\]
where the integer $M(\e)$ is defined as in the proof of Proposition \ref{p:new}.

Let $d \geqs N(c,\delta)$ be an integer and set $d_1 = (1+\delta)d/e$. Let $G \leqs {\rm Sym}(\O)$ be a permutation group of degree $n$ such that $G_{(\Delta)} \in \Gamma_d$ for every subset $\Delta \subseteq \O$ of size $c$. Our goal is to establish the bound $|G| \leqs d_1^{n-1}$.

We will first handle the case $n < d$.  If $n \leqs d/e$ then
\[
|G| \leqs n! \leqs n^{n-1} \leqs (d/e)^{n-1} < ((1+\delta)d/e)^{n-1} = d_1^{n-1},
\]
as required. Otherwise we have $d/e < n < d$. This yields
\[
|G| \leqs n! <  e n^{1/2} (n/e)^n < e d^{1/2} (d/e)^n < ((1+\delta)d/e)^{n-1},
\]
where the last inequality follows from the fact that $d \geqs M(\e)$, which yields
\[
d^{3/2} \leqs (1 + \e)^{d/e-1} < (1 + \delta)^{d/e-1}.
\]

To complete the proof, we may assume that $n \geqs d$.
First suppose $G$ is transitive and fix $\a \in \O$. Set $H = G_{\a}$, so $|G:H| = n$ and we may regard $H$ as a permutation group on
$\O \setminus \{ \a \}$ of degree $n-1$. In addition, the pointwise stabilizers in $H$ of subsets of size $c-1$ are contained in $\Gamma_d$.

Set $d_2 = (1+\e)d/e$. Recall that $d \geqs N(c, \delta) \geqs N(c-1, \e)$, so the inductive hypothesis yields $|H| \leqs d_2^{n-2}$ and
\[
|G| = n |H| \leqs n d_2^{n-2}.
\]
Since $n \geqs d \geqs N(c, \delta) \geqs M(\e)$ we have $n^{3/2} \leqs (1 + \e)^{n/e-1}$, hence
\[
n \leqs (1+\e)^{\frac{2}{3}(n/e-1)}.
\]
We conclude that
\[
|G| \leqs (1+\e)^{\frac{2}{3}(n/e-1)} d_2^{n-2} < (1+\e)^{\frac{2}{3}(n/e-1)} ((1+\e)d/e)^{n-1} < (1+\e)^{a n - 5/3}(d/e)^{n-1},
\]
where $a = 1 + 2/(3e) < 5/3$. It follows that $|G| < (1+\e)^{5(n-1)/3} (d/e)^{n-1}$, and since
$(1+\e)^{5/3} = 1+\delta$ we obtain
\[
|G| < ((1+\delta)d/e)^{n-1} = d_1^{n-1},
\]
as required.

Finally, suppose $G$ is intransitive. Let $\O_1 \subseteq \O$ be an orbit of $G$ and let $\O_2 = \O \setminus \O_1$. For $i = 1,2$, let $N_i \normeq G$ be the kernel of the action of $G$ on $\O_i$ and set $G_i = G/N_i \leqs \Sym (\O_i)$.
Note that if  $\Delta \subseteq \O_i$ is a subset of size $c$ then $G_{(\Delta)} \geqs N_i$ and
$(G_i)_{(\Delta)} = G_{(\Delta)}/N_i \in \Gamma_d$. In addition, observe that $G \leqs G_1 \times G_2$.

Let $k = |\O_1|$, so $n-k = |\O_2|$ and $k, n-k < n$. Recall that $d \geqs N(c,\delta) \geqs c$. Thus if $c>k$ then $k < d$, so as in the case $n<d$ discussed above we deduce that $|G_1| \leqs k! < d_1^{k-1}$. Similarly, $|G_2| \leqs (n-k)! <d_1^{n-k-1}$ if $c>n-k$. Now assume $c \leqs k$. As noted above, the pointwise stabilizer of every $c$-element subset of $\O_1$ is contained in $\Gamma_d$, so by applying induction on $n$ (with $c$ and $\delta$ fixed), we deduce that $|G_1|<d_1^{k-1}$. Similarly, if $c \leqs n-k$ then $|G_2|<d_1^{n-k-1}$. We conclude that
\[
|G| \leqs |G_1||G_2| < d_1^{k-1} d_1^{n-k-1} < d_1^{n-1}
\]
and the result follows.
\end{proof}

\vs

This completes the proof of Theorem \ref{t:3}. In addition, as noted in Section \ref{s:intro}, Corollary \ref{t:2} follows immediately.

\section{Proof of Theorem \ref{t:4}}\label{s:proof2}

In this section we prove Theorem \ref{t:4}. Let $G \leqs {\rm Sym}(\O)$ be a finite primitive permutation group with point stabilizer $H$ and assume $G$ is not an affine type group. In addition, let us assume every $2$-point stabilizer in $G$ is solvable.
Our goal is to show that $b(G) \leqs 6$, with equality when $G$ is almost simple if and only if $G = {\rm Sp}_{6}(2)$ and $H = {\rm O}_{6}^{+}(2)$. If $H$ itself is solvable, then the main theorem of \cite{Bur20} implies that $b(G) \leqs 5$, so we may assume for the remainder that $H$ is nonsolvable.

\subsection{Almost simple groups}\label{ss:as}

Here we prove Theorem \ref{t:4} in the case where $G$ is almost simple with socle $G_0$. We begin by handling the groups with socle a sporadic simple group.

\begin{prop}\label{p:6spor}
If $G_0$ is a sporadic simple group, then $b(G) \leqs 5$.
\end{prop}

\begin{proof}
Suppose $b(G) \geqs 6$. Then the main theorem of \cite{BOW} implies that $(G,H)$ is one of the following:
\[
({\rm M}_{24},{\rm M}_{23}),\; ({\rm M}_{23},{\rm M}_{22}),\; ({\rm Co}_{3}, {\rm McL}.2),\; ({\rm Co}_{2}, {\rm U}_{6}(2).2), \; ({\rm Fi}_{22}.2, 2.{\rm U}_{6}(2).2).
\]
In the first three cases we note that $G$ is $2$-transitive and the $2$-point stabilizers are ${\rm M}_{22}$, ${\rm L}_{3}(4)$ and ${\rm U}_{4}(3).2$, respectively. Similarly, in the latter two cases it is easy to check that every $2$-point is nonsolvable and the result follows.
\end{proof}

\begin{prop}\label{p:6alt}
If $G_0 = A_m$ is an alternating group, then $b(G) \leqs 5$.
\end{prop}

\begin{proof}
The cases with $m \leqs 12$ can be checked using {\sc Magma} \cite{magma}, so we will assume $m \geqs 13$. In particular, $G = A_m$ or $S_m$. If $H$ acts primitively on $\{1, \ldots, m\}$ then $b(G)=2$ by the main theorem of \cite{BGS}. Therefore we may assume that either
\begin{itemize}\addtolength{\itemsep}{0.2\baselineskip}
\item[{\rm (a)}] $H = (S_k \times S_{m-k}) \cap G$ for some $1 \leqs k < m/2$; or
\item[{\rm (b)}] $H = (S_{k} \wr S_{m/k}) \cap G$, where $1 < k \leqs m/2$ and $k$ divides $m$.
\end{itemize}

In case (a), we identify $\O$ with the set of $k$-element subsets of $\{1, \ldots, m\}$. If we take the $k$-sets $\a = \{1, \ldots, k\}$ and $\b = \{1, \ldots, k-1,k+1\}$, then $G_{\a,\b} = (S_{k-1} \times S_{m-k-1}) \cap G$ is a $2$-point stabilizer and thus $k$ and $m-k$ are at most $5$ (since we are assuming that $G_{\a,\b}$ is solvable). But this is incompatible with the bound $m \geqs 13$.

Finally, let us turn to case (b). Here we identify $\O$ with the set of partitions of $\{1,\ldots, m\}$ into $m/k$ parts of size $k$. Consider the partitions $\a,\b \in \O$, where
\begin{align*}
\a & = \left\{ \{1, \ldots, k\}, \{k+1, \ldots, 2k\}, \ldots \right\} \\
\b & = \left\{ \{1, \ldots, k-1,k+1\}, \{k,k+2, \ldots, 2k\}, \ldots \right\}
\end{align*}
only differ in the first two parts. Visibly, the $2$-point stabilizer $G_{\a,\b}$ contains $A_{k-\e}$, where $\e=0$ if $m/k \geqs 3$, otherwise $\e=1$. Therefore, since $m \geqs 13$, we may assume $k \in \{2,3,4\}$ and $m/k \geqs 5$ (note that $H$ is solvable if $m=16$ and $k=4$). In particular, $k<m/k$ and thus \cite[Theorem 2]{BGL} (also see Remark 2.8 in \cite{BGL} for $G=A_m$) gives $b(G) \leqs 4$.
\end{proof}

\begin{prop}\label{p:6ex}
If $G_0$ is an exceptional group of Lie type, then $b(G) \leqs 5$.
\end{prop}

\begin{proof}
Suppose $b(G) \geqs 6$. By the main theorem of \cite{Bur18}, it follows that $b(G)=6$ and either
\begin{itemize}\addtolength{\itemsep}{0.2\baselineskip}
\item[{\rm (a)}] $G_0 = E_7$ and $H = P_7$; or
\item[{\rm (b)}] $G_0 = E_6$ and $H = P_1$ or $P_6$.
\end{itemize}
Here $P_k$ is the standard notation for a maximal parabolic subgroup of $G$ corresponding to the $k$-th node in the Dynkin diagram of $G_0$ (where we adopt the standard Bourbaki \cite{Bou} labelling of nodes).

Set $H_0 = H \cap G_0$ and let $H_0 = QL$ be a Levi decomposition. In each of the above cases, it is easy to see that $L \cap L^x$ is nonsolvable for some $x \in G_0$, which implies that the $2$-point stabilizer $H \cap H^x$ is nonsolvable. For example, consider case (a) and let $\{\a_1, \ldots, \a_7\}$ be a set of simple roots for the corresponding root system of $G_0$. Let $x \in G_0 \setminus H_0$ be a long root element in the root subgroup $U_{\a_7}$ of $G_0$. Then $x$ centralizes the nonsolvable subgroup $\la U_{\pm \a_i} \,:\, 1 \leqs i \leqs 5 \ra < L$ of type $D_5$ and thus $L \cap L^x$ is nonsolvable. Similar reasoning applies in case (b).
\end{proof}

\begin{prop}\label{p:6class}
If $G_0$ is a classical group, then $b(G) \leqs 6$, with equality if and only if $G = {\rm Sp}_{6}(2)$ and $H = {\rm O}_{6}^{+}(2)$.
\end{prop}

\begin{proof}
Let $G_0$ be a finite simple classical group over $\mathbb{F}_q$ with natural module $V$ of dimension $m$. Write $q=p^f$, where $p$ is a prime, and set $H_0 = H \cap G_0$. As in the proof of Proposition \ref{p:as}, we say that $G$ is \emph{standard} if
\begin{itemize}\addtolength{\itemsep}{0.2\baselineskip}
\item[{\rm (i)}] $H$ is the stabilizer in $G$ of a subspace (or pair of subspaces) of $V$; or
\item[{\rm (ii)}] $G_0 = {\rm Sp}_{m}(q)$, $q$ is even and $H_0 = {\rm O}_{m}^{\pm}(q)$.
\end{itemize}
Otherwise, $G$ is \emph{nonstandard}. By the main theorem of \cite{Bur07} we have $b(G) \leqs 5$ if $G$ is nonstandard, so for the remainder of the proof we may assume $G$ is standard. Let us also recall that we may assume $H$ is nonsolvable.

\vs

\noindent \emph{Case 1. Linear groups.}

\vs

To begin with, let us assume $G_0 = {\rm L}_{m}(q)$ and fix a basis $\{e_1, \ldots, e_m\}$ for $V$. There are three cases to consider:
\begin{itemize}\addtolength{\itemsep}{0.2\baselineskip}
\item[{\rm (a)}] $H$ is a parabolic subgroup of type $P_k$ with $1 \leqs k \leqs m/2$;
\item[{\rm (b)}] $H$ is a parabolic subgroup of type $P_{k,m-k}$ with $1 \leqs k < m/2$;
\item[{\rm (c)}] $H$ is of type ${\rm GL}_{k}(q) \times {\rm GL}_{m-k}(q)$ with $1 \leqs k < m/2$.
\end{itemize}

Suppose we are in case (a), which allows us to identify $\O$ with the set of $k$-dimensional subspaces of $V$. Set $\a = \la e_1, \ldots, e_k\ra$ and $\b = \la e_{k+1}, \ldots, e_{2k} \ra$. If $k \geqs 3$, or if $k=2$ and $q \geqs 4$, then  the $2$-point stabilizer $G_{\a,\b}$ has a nonabelian composition factor ${\rm L}_{k}(q)$ and is therefore nonsolvable. Next assume $k=2$ and $q \leqs 3$. Set $\b' = \la e_2, e_3\ra$ and observe that ${\rm L}_{m-3}(q)$ is a composition factor of $G_{\a,\b'}$ if $m \geqs 6$. If $m=5$ then a straightforward {\sc Magma} computation gives  $b(G)=4$, whereas $H$ is solvable if $m=4$. Finally, let us assume $k=1$. If $m \geqs 5$ then $G_{\a,\b}$ has a composition factor ${\rm L}_{m-2}(q)$. Similarly, if $m=4$ then we may assume $q \leqs 3$ and we calculate that $b(G) = 4+\delta_{3,q}$. For $m = 3$ we have $G \leqs {\rm P \Gamma L}_{3}(q)$ and it is easy to check that $\{ \a,\b, \la e_3 \ra, \la e_1+e_2+e_3\ra,\la e_1+e_2+\mu e_3\ra \}$ is a base for $G$, where $\mathbb{F}_{q}^{\times} = \la \mu \ra$. Finally, we note that $H$ is solvable if $m=2$.

Next we turn to case (b). Here we identify $\O$ with the set of flags of $V$ of the form $0 < U < W < V$, where $\dim U = k$ and $\dim W = m-k$. Let $\a,\b \in \O$ be the flags
\[
0 < \la e_1, \ldots, e_k \ra < \la e_1, \ldots, e_{m-k}\ra < V,\;\; 0 < \la e_{m-k+1}, \ldots, e_m \ra < \la e_{k+1}, \ldots, e_{m}\ra < V
\]
respectively. Then $G_{\a,\b}$ contains a Levi factor of the parabolic subgroup $H = G_{\a}$ and thus every $2$-point stabilizer is solvable if and only if $H$ is solvable.

To complete the proof of the proposition for $G_0 = {\rm L}_{m}(q)$, we may assume $H$ is of type ${\rm GL}_{k}(q) \times {\rm GL}_{m-k}(q)$ with $1 \leqs k < m/2$. Here we identify $\O$ with the set of direct sum decompositions $V = U \oplus W$ with $\dim U = k$. Define $\a,\b \in \O$ as follows
\[
\la e_1, \ldots, e_k \ra \oplus \la e_{k+1}, \ldots, e_m\ra,\;\; \la e_{k+1}, \ldots, e_{2k} \ra \oplus \la e_1, \ldots, e_k,e_{2k+1}, \ldots, e_m\ra.
\]
Clearly, if $k \geqs 3$, or if $k=2$ and $q \geqs 4$, then $G_{\a,\b}$ has a composition factor ${\rm L}_{k}(q)$. Next assume $k = 2$ and $q \leqs 3$. If $m=5$ then a {\sc Magma} computation shows that $b(G) = 3$, otherwise $G_{\a,\b'}$ has a composition factor ${\rm L}_{m-3}(q)$, where $\b' \in \O$ is the decomposition $\la e_1, e_3\ra \oplus \la e_2,e_4, \ldots, e_m\ra$. Finally, suppose $k=1$. If $m \geqs 5$ then $G_{\a,\b}$ has a composition factor ${\rm L}_{m-2}(q)$. Similarly, if $m=4$ then we may assume $q \leqs 3$ and we calculate that $b(G) \leqs 4$. Finally, if $m=3$ then one can check that $\{\a,\b,\gamma,\delta\}$ is a base for $G$, where $\gamma, \delta \in \O$ are the decompositions $\la e_1+e_2+e_3\ra \oplus \la e_2,e_3\ra$ and $\la e_1+e_2+\mu e_3\ra \oplus \la e_1, e_2 \ra$ with $\mathbb{F}_{q}^{\times} = \la \mu \ra$. Indeed, it is easy to verify that the pointwise stabilizer of $\{\a,\b,\gamma,\delta\}$ in ${\rm P \Gamma L}_{3}(q)$ is trivial and we observe that $\delta$ is not fixed by the inverse-transpose graph automorphism of $G_0$, which maps $\la e_1+e_2+\mu e_3\ra$ to the $2$-space $\la e_1-e_2, \mu e_1 - e_3\ra$.

\vs

\noindent \emph{Case 2. Unitary groups.}

\vs

Now assume $G_0 = {\rm U}_{m}(q)$. Following  \cite[Proposition 2.3.2]{KL}, fix a standard basis
\[
\left\{\begin{array}{ll}
\{e_1, \ldots, e_{\ell},f_1, \ldots, f_{\ell}\} & \mbox{if $m=2\ell$} \\
\{e_1, \ldots, e_{\ell},f_1, \ldots, f_{\ell},x\} & \mbox{if $m=2\ell+1$.}
\end{array}\right.
\]

First assume $H$ is a parabolic subgroup of type $P_k$ with $1 \leqs k \leqs m/2$, so we may identify $\O$ with the set of totally isotropic $k$-dimensional subspaces of $V$. Set $\a = \la e_1, \ldots, e_k\ra$ and $\b = \la f_1, \ldots, f_k\ra$ with respect to the above basis. If $k \geqs 2$ then $G_{\a,\b}$ has a composition factor ${\rm L}_{k}(q^2)$, so we may assume $k=1$. If $m=3$, or $m = 4$ and $q \leqs 3$, or $(m,q) = (5,2)$ then $H$ is solvable, whereas ${\rm U}_{m-2}(q)$ is a composition factor of $G_{\a,\b}$ in the remaining cases.

Now let us turn to the case where $H = G_{\a}$ is the stabilizer of a nondegenerate $k$-space, where $1 \leqs k < m/2$. Fix $\b \in \O$ such that $\a \perp \b$ is a  nondegenerate $2k$-space. If $k \geqs 3$ and $(k,q) \ne (3,2)$, or if $k=2$ and $q \geqs 4$, then $G_{\a,\b}$ has a composition factor ${\rm U}_{k}(q)$. Suppose $(k,q) = (3,2)$, so $m \geqs 7$. If $m \geqs 10$ then $G_{\a,\b}$ has a composition factor ${\rm U}_{m-6}(2)$. Similarly, if $m =9$ then ${\rm U}_{4}(2)$ is a composition factor of $G_{\a',\b'}$, where $\a' = \la e_1,f_1,x\ra$ and $\b' = \la e_2,f_2,x\ra$, and for $m \in \{7,8\}$ one checks that $b(G) \leqs 3$. Next assume $k=2$ and $q \leqs 3$. If $m \geqs 7$  and $(m,q) \ne (7,2)$ then ${\rm U}_{m-4}(q)$ is a composition factor of $G_{\a,\b}$ and one checks that $b(G) \leqs 4$ if $m \in \{5,6\}$ or if $(m,q) = (7,2)$.

Finally, suppose $k=1$ and note that $H$ is solvable if $m=3$. For $m \geqs 4$ we see that $G_{\a,\b}$ is solvable if and only if $(m,q) = (5,2)$ or $m=4$ and $q \leqs 3$; in each of these cases, one can use {\sc Magma} to verify the bound $b(G) \leqs 5$.

\vs

\noindent \emph{Case 3. Symplectic groups.}

\vs

Let $G_0 = {\rm PSp}_{m}(q)'$ with $m \geqs 4$ and fix a standard basis $\{e_1, \ldots, e_{m/2}, f_1, \ldots, f_{m/2}\}$ for $V$ (see \cite[Proposition 2.4.1]{KL}). In view of the isomorphisms ${\rm PSp}_{4}(2)' \cong A_6$ and ${\rm PSp}_{4}(3) \cong {\rm U}_{4}(2)$, we may assume $q \geqs 4$ if $m=4$.

First assume $H$ is a parabolic subgroup of type $P_k$, where $1 \leqs k \leqs m/2$. Set $\a = \la e_1, \ldots, e_k \ra$ and $\b = \la f_1, \ldots, f_k \ra$. If $k=1$  then $G_{\a,\b}$ has a composition factor ${\rm PSp}_{m-2}(q)'$. Similarly, if $k \geqs 3$, or if $k=2$ and $q \geqs 4$ then ${\rm L}_{k}(q)$ is a composition factor of $G_{\a,\b}$. Finally, assume $k=2$ and $q \leqs 3$. If $m \geqs 8$ then ${\rm PSp}_{m-4}(q)'$ is a composition factor of $G_{\a,\b}$, while for $m=6$ it is easy to check that $b(G) \leqs 4$.

Now assume $H$ is the stabilizer of a nondegenerate $k$-space, where $2 \leqs k < m/2$ is even. Set $\a = \la e_1, \ldots, e_{k/2}, f_1, \ldots, f_{k/2}\ra$ and $\b = \la e_{k/2+1}, \ldots, e_{k}, f_{k/2+1}, \ldots, f_{k}\ra$. If $k \geqs 4$, or if $k=2$ and $q \geqs 4$ then ${\rm PSp}_{k}(q)'$ is a composition factor of $G_{\a,\b}$. Similarly, if $k=2$ and $q \leqs 3$ then we quickly reduce to the case $m=6$, where a straightforward {\sc Magma} computation gives $b(G) \leqs 4$.

To complete the analysis of symplectic groups, we may assume $q$ is even and $H$ is of type ${\rm O}_{m}^{\e}(q)$. Note that if we regard $G_0$ as the isomorphic orthogonal group ${\rm O}_{m+1}(q)$, then we may identify $\O$ with the set of nondegenerate $m$-dimensional subspaces of type $\e$ of the natural module $W$ for ${\rm O}_{m+1}(q)$ (recall that a nondegenerate $2\ell$-dimensional subspace of an orthogonal space is of \emph{plus-type} if it contains a totally singular $\ell$-space, otherwise it is a \emph{minus-type} space). Here ${\rm O}_{m+1}(q)$ is the isometry group of a nonsingular quadratic form $Q$ on $W$ with a $1$-dimensional radical $\la v \ra$ (note that $v$ is nonsingular and we may assume $Q(v)=1$). An $m$-dimensional subspace of $W$ is nondegenerate if and only if it does not contain $v$.

Fix a basis
\[
\{e_1, \ldots, e_{m/2-1}, f_1, \ldots, f_{m/2-1},x,y,v\}
\]
for $W$, where $\la e_1, \ldots, e_{m/2-1},f_1, \ldots, f_{m/2-1}\ra$ and $\la x,y \ra$ are nondegenerate spaces of plus-type and $\e$-type, respectively. Define $\a,\b \in \O$ by setting
\begin{align*}
\a & = \la e_1, \ldots, e_{m/2-1}, f_1, \ldots, f_{m/2-1},x,y\ra \\
\b & =  \la e_1+v, e_{2}, \ldots, e_{m/2-1}, f_1, \ldots, f_{m/2-1},x,y\ra.
\end{align*}
Then the pointwise stabilizer $(G_0)_{\a,\b}$ visibly has a section isomorphic to $\O_{m-2}^{\e}(q)$ (acting trivially on $\la e_1,f_1,v\ra$), so we may assume that either $m=4$ or $(m,q,\e) = (6,2,+)$.
In the latter case, $G$ is $2$-transitive, $G_{\a,\b} = 2^4.{\rm L}_{2}(2)^2.2$ is solvable and with the aid of {\sc Magma} one checks that $b(G) = 6$ (this is the special case recorded in the statement of the proposition).

Finally, let us assume $m=4$, so $q \geqs 4$. We claim that there exists $\a,\b \in \O$ such that $(G_0)_{\a,\b}$ has a section isomorphic to ${\rm L}_{2}(q)$. To see this, first assume $\e=+$, in which case $\la x,y \ra$ is a plus-type $2$-space and we may assume that $Q(x) = Q(y) = 0$ and $Q(x+y) = 1$. Choose $\xi \in \mathbb{F}_q^{\times}$ such that the polynomial $t^2+t+\xi^2 \in \mathbb{F}_q[t]$ is reducible. Then the $4$-spaces 
\[
\a = \la e_1, f_1, x, y \ra,\;\; \b = \la e_1,f_1,x+y,x+\xi v\ra
\]
are contained in $\O$ and by considering the common $3$-space $\la e_1,f_1,x+y\ra$ we deduce that
${\rm O}_{3}(q) \leqs (G_0)_{\a,\b}$. A similar argument applies when $\e=-$. Here $\la x,y \ra$ is a minus-type $2$-space and we are free to assume that $Q(x) = 1$ and $Q(y) = Q(x+y)=\xi$, where $t^2+t+\xi \in \mathbb{F}_q[t]$ is irreducible. Choose $\l \in \mathbb{F}_q^{\times}$ so that $t^2+t+\xi+\l^2 \in \mathbb{F}_q[t]$ is also irreducible. Then 
\[
\a = \la e_1,f_1,x,y \ra,\;\; \b = \la e_1,f_1,x, y+\l v \ra
\]
are spaces in $\O$ and once again we have ${\rm O}_{3}(q) \leqs (G_0)_{\a,\b}$. The result follows.

\vs

\noindent \emph{Case 4. Odd-dimensional orthogonal groups.}

\vs

Next we assume $G_0 = \O_m(q)$, where $mq$ is odd and $m \geqs 7$. Write $m=2\ell+1$ and fix a standard basis
$\{e_1, \ldots, e_{\ell}, f_1, \ldots, f_{\ell},x\}$
for $V$ as in \cite[Proposition 2.5.3(iii)]{KL}.

First let $H$ be a parabolic subgroup of type $P_k$ with $1 \leqs k < m/2$. With respect to a standard basis for $V$, set $\a = \la e_1, \ldots, e_k\ra$ and $\b = \la f_1, \ldots, f_k \ra$. Then $G_{\a,\b}$ has a composition factor ${\rm L}_{k}(q)$ if $k \geqs 3$, and $\O_{m-2k}(q)$ if $k \leqs 2$ and $(m,k,q) \ne (7,2,3)$. Finally, we note that $H$ is solvable when $(m,k,q) = (7,2,3)$.

Now suppose $H$ is the stabilizer of a nondegenerate $k$-space of type $\e$ with $1 \leqs k < m/2$. Note that if $k$ is odd then $G$ has two orbits on the set of all nondegenerate $k$-spaces; the spaces in the two orbits are distinguished by the discriminant of the restriction of the defining quadratic form on $V$, which is either a square or nonsquare (see \cite[p.32]{KL}). The respective actions of $G$ are permutation isomorphic, so without loss of generality we may assume that if $k$ is odd then $\O$ is the set of nondegenerate $k$-spaces with square discriminant. On the other hand, if $k$ is even then we identify $\O$ with the set of all nondegenerate $k$-spaces of type $\e = \pm$.

Fix a nondegenerate $k$-space $\a \in \O$ and choose $\b \in \O$ such that $\a\perp \b$ is a nondegenerate $2k$-space. If $k \geqs 5$, then ${\rm P\O}_{k}^{\e}(q)$ is a composition factor of $G_{\a,\b}$. Similarly, if $k \in \{3,4\}$ then we may assume $q=3$, with $\e=+$ if $k=4$. Suppose $(k,q,\e) = (4,3,+)$. If $m \geqs 13$ then $G_{\a,\b}$ has a composition factor $\O_{m-8}(3)$, so we can assume $m \in \{9,11\}$. If $m=11$ then $\O_5(3)$ is a composition factor of $G_{\a,\b'}$, where $\a = \la e_1,e_2,f_1,f_2\ra$ and $\b' = \la e_1,e_3,f_1,f_3\ra$, and for $m=9$ we calculate that $b(G)=2$.

To complete the analysis we may assume $k \in \{1,2\}$ or $k=q=3$. Suppose $k=q=3$ and set $\a = \la e_1,f_1,x\ra$ and $\b' = \la e_2,f_2,x\ra$. If $m \geqs 11$ then $G_{\a,\b'}$ has a composition factor of type $\O_{m-5}^{\e'}(3)$, while a {\sc Magma} computation gives $b(G) \leqs 3$ if $m \in \{7,9\}$. Finally, if $k \in \{1,2\}$ then $G_{\a,\b}$ has a composition factor $\O_{m-2k}(q)$ if $(m,k,q) \ne (7,2,3)$, otherwise $\e=-$ (see \cite[Table 8.39]{BHR}) and $b(G)=3$.

\vs

\noindent \emph{Case 5. Even-dimensional orthogonal groups.}

\vs

To complete the proof of the proposition, we may assume $G_0 = {\rm P\O}_{m}^{\e}(q)$ with $m \geqs 8$ even. Write $m=2\ell$ and fix standard bases
\[
\left\{\begin{array}{ll}
\{e_1, \ldots, e_{\ell}, f_1, \ldots, f_{\ell}\} & \mbox{if $\e=+$} \\
\{e_1, \ldots, e_{\ell-1}, f_1, \ldots, f_{\ell-1},x,y\} & \mbox{if $\e=-$}
\end{array}\right.
\]
as in \cite[Proposition 2.5.3]{KL}.

To begin with, let us assume $H$ is a parabolic subgroup of type $P_k$. If $k<m/2$ then we may identify $\O$ with the set of all totally singular $k$-dimensional subspaces of $V$. However, if $k=m/2$ then $\e=+$ and $G_0$ has two orbits on the set of totally singular $k$-dimensional subspaces of $V$ (see \cite[Lemma 2.5.8]{KL}). Without loss of generality, in the latter case we may assume that $\O$ is the set of totally singular $k$-spaces $\b$ such that $k - \dim (\a \cap \b)$ is even (see \cite[Lemma 2.5.8]{KL}), where $\a = \la e_1, \ldots, e_k \ra$.

First assume $k=m/2$, so $\e=+$. If $k$ is odd then $G_{\a,\b}$ has a composition factor ${\rm L}_{k-1}(q)$, where $\b = \la e_1,f_2, \ldots, f_k\ra$. Similarly, if $k$ is even and we set $\b = \la f_1, \ldots, f_k\ra$ then ${\rm L}_{k}(q)$ is a composition factor of $G_{\a,\b}$. The same conclusion holds (with $\a$ and $\b$ defined in the same way) if $(\e,k) \ne (+,m/2)$ and either $k \geqs 3$, or $k=2$ and $q \geqs 4$. If $k=1$ then $G_{\a,\b}$ has a composition factor ${\rm P\O}^{\e}_{m-2}(q)$. Finally, suppose $k=2$ and $q \leqs 3$. If $m \geqs 10$ or $(m,\e)=(8,-)$ then ${\rm P\O}_{m-4}^{\e}(q)$ is a composition factor of $G_{\a,\b}$, whereas $H$ is solvable if $(m,\e) = (8,+)$.

Next suppose $H$ is the stabilizer of a nondegenerate $k$-space of type $\e'$, where $1 \leqs k \leqs m/2$. Note that if $k$ is odd then $q$ is odd and $G$ has two orbits on the set of nondegenerate $k$-dimensional subspaces of $V$; as in Case 4, we may assume $\O$ comprises the spaces with square discriminant. For $k$ even, we identify $\O$ with the set of all nondegenerate $k$-spaces of type $\e'$. In addition, observe that  $k=m/2$ only if $m \equiv 0 \imod{4}$ and $\e=-$ (see \cite[Tables 3.5.E, 3.5.F]{KL}, for example), in which case we are free to assume that $\e'=+$.

Let us begin by handling the special case where $k=m/2$, so $m \equiv 0 \imod{4}$, $\e=-$ and $\e'=+$. Fix $\a = \la e_1, \ldots, e_{m/4}, f_1, \ldots, f_{m/4}\ra \in \O$. If we set
\[
\b = \la e_1, \ldots, e_{m/4-1},e_{m/4+1}, f_1, \ldots, f_{m/4-1}, f_{m/4+1}\ra \in \O
\]
then $G_{\a,\b}$ has a composition factor ${\rm P\O}_{m/2-2}^{-}(q)$ if $m \geqs 12$. Now assume $m=8$. If $q \leqs 3$ then a calculation with {\sc Magma} shows that $b(G) \leqs 3$. On the other hand if $q \geqs 4$ and $\b'=\la e_1,e_2,f_1+x,f_2+x\ra$, then one checks that $G_{\a,\b'}$ has a section isomorphic to ${\rm L}_{2}(q)$.

For the remainder, we may assume $k < m/2$ and $\O$ is the set of nondegenerate $k$-dimensional subspaces of $V$ of type $\e'$. Fix $\a,\b \in \O$ such that $\a \perp \b$ is a nondegenerate $2k$-space (note that this space has plus-type if $k$ is even). If $k \geqs 5$ then $G_{\a,\b}$ has a composition factor ${\rm P\O}_{k}^{\e'}(q)$, so we may assume $k \leqs 4$. Similarly, if $k=4$ (with $q \geqs 4$ if $\e'=+$) then $G_{\a,\b}$ is nonsolvable, so let us assume $(k,\e')=(4,+)$ and $q \leqs 3$. Set $\a = \la e_1,e_2,f_1,f_2\ra$ and $\b' = \la e_1,e_3,f_1,f_3\ra$. Then $G_{\a,\b'}$ has a composition factor ${\rm P\O}_{m-6}^{\e}(q)$ unless $(m,\e) = (10,+)$, in which case a {\sc Magma} computation gives $b(G) \leqs 3$.

Next suppose $k=3$, so $q$ is odd. If $q \geqs 5$, then $G_{\a,\b}$ has a composition factor $\O_{3}(q)$, so we may assume $q=3$. If $m \geqs 12$ then ${\rm P\O}_{m-6}^{\e''}(q)$ is a composition factor of $G_{\a,\b}$ (the sign $\e''$ depends on the type of the nondegenerate $6$-space $\a \perp \b$). If $m \in \{8,10\}$ then with the aid of {\sc Magma} one checks that $b(G) \leqs 3$.

Finally, let us assume $k \in \{1,2\}$. If $k=1$, or if $k=2$ and $(m,\e) \ne (8,+)$, then $G_{\a,\b}$ has a composition factor ${\rm P\O}_{m-2k}^{\e''}(q)$, where $\e''=\e$ if $k=2$. Similarly, if $(m,k,\e) = (8,2,+)$ and $q \geqs 4$ then $G_{\a,\b}$ is nonsolvable (it has an ${\rm L}_{2}(q)$ composition factor). Finally, if $(m,k,\e) = (8,2,+)$ and $q \leqs 3$ then a {\sc Magma} computation gives $b(G) \leqs 4$.

To complete the proof, we may assume $q$ is even and $H$ is the stabilizer of a nonsingular $1$-space. By identifying $\O$ with the set of all such subspaces of $V$, we may choose $\a,\b \in \O$ so that $\a \oplus \b$ is a nondegenerate $2$-space of  minus-type  and thus $\O_{m-2}^{-\e}(q)$ is a composition factor of $G_{\a,\b}$.
\end{proof}

\vs

This completes the proof of Theorem \ref{t:4} for almost simple groups.

\subsection{Diagonal type groups}

\begin{prop}\label{p:fn3}
Let $G \leqs {\rm Sym}(\O)$ be a finite primitive group of diagonal type with point stabilizer $H$. If every $2$-point stabilizer in $G$ is solvable, then $b(G) \leqs 4$.
\end{prop}

\begin{proof}
Let $N= T^k$ be the socle of $G$ and write $N \leqs G \leqs N.({\rm Out}(T) \times S_k)$, where $T$ is a nonabelian simple group. Let $P_G$ be the subgroup of $S_k$ induced by the conjugation action of $G$ on the set of $k$ factors of $N$. The primitivity of $G$ implies that $G = NH$, so $H$ also induces $P_G$ on the factors of $N$. We may assume that
\[
H = \{ (a, \ldots, a)\pi \in G \,:\, a \in {\rm Aut}(T), \pi \in S_k \}.
\]

If $P_G \ne A_k, S_k$ then \cite[Theorem 1.1]{fawcett} gives $b(G) = 2$, so we may assume $P_G$ contains $A_k$. Here \cite[Theorem 1.2]{fawcett} gives
\begin{equation}\label{e:diag}
b(G) \leqs \max\left\{4, \left\lceil \frac{\log k}{\log |T|}\right\rceil+2\right\},
\end{equation}
which is at most $4$ if $k \leqs |T|^2$. For the remainder, we may assume that $k > |T|^2$ and $P_G = A_k$ or $S_k$.  Fix $1\ne s \in T$ and set $g = (s,1,\ldots, 1) \in N$, so
\[
H \cap H^g = \{(a, \ldots, a)\pi \in G \,:\, a \in C_{{\rm Aut}(T)}(s), \pi \in S_k, 1^{\pi}=1 \}
\]
is a $2$-point stabilizer. Let $r$ be the largest prime less than $|T|^2$. Since $|{\rm Out}(T)| < |T|$, it follows that $(r, |{\rm Out}(T)|) = 1$ and thus \cite[Lemma 3.11]{fawcett} implies that $N{:}A_k \leqs G$. Therefore, $A_{k-1} \leqs H \cap H^g$ and we conclude that $H \cap H^g$ is nonsolvable.
\end{proof}

\begin{rem}
The bound in Proposition \ref{p:fn3} is best possible. For example, $b(G) = 4$ for the diagonal type group $G = A_5^2.2^2$ of degree $60$ and we calculate that every $2$-point stabilizer in this group is solvable (indeed, the largest $2$-point stabilizer has order $16$).
\end{rem}

\subsection{Product type groups}

\begin{prop}\label{p:fn4}
Let $G \leqs {\rm Sym}(\O)$ be a finite primitive group of product type with point stabilizer $H$. If every $2$-point stabilizer in $G$ is solvable, then $b(G) \leqs 6$.
\end{prop}

\begin{proof}
Here we adopt the same notation as in the proof of Proposition \ref{p:prod}. In particular, $N \leqs G \leqs L \wr P$, where $N = T^k$ is the socle of $G$ and $L \leqs {\rm Sym}(\Delta)$ is a primitive group with socle $T$, which is either almost simple or diagonal type. In addition, $P \leqs S_k$ is the transitive group induced by the conjugation action of $G$ on the $k$ factors of $N$ and we write $\O = \Delta_1 \times \cdots \times \Delta_k = \Delta^k$.

\vs

\noindent \emph{Case 1. $L$ is almost simple.}

\vs

First assume $L$ is almost simple and fix $\delta_1, \delta_2 \in \Delta$ with $\delta_1 \ne \delta_2$. Set $\a_i = (\delta_i, \ldots, \delta_i) \in \O$ for $i=1,2$. Then $(T_{\delta_1,\delta_2})^k \leqs G_{\a_1,\a_2}$, so $T_{\delta_1,\delta_2}$ is solvable and we deduce that $L_{\delta_1,\delta_2}$ is solvable (recall that $L/T \leqs {\rm Out}(T)$ is solvable). Therefore, every $2$-point stabilizer in $L$ is solvable and thus $b(L) \leqs 6$ from our work in Section \ref{ss:as}. By repeating the argument in the proof of Proposition \ref{p:prod} we see that
\[
Q = \{ (x_1, \ldots, x_k)\pi  \in G \,:\, \mbox{$x_i=1$ for all $i$} \}
\]
is the pointwise stabilizer in $G$ of a specific set of $b(L)$ elements in $\O$. Since $b(L) \geqs 2$ it follows that $Q$ is solvable and thus Theorem \ref{t:dist}(i) implies that $d(Q) \leqs 5$. Then by applying the upper bound in \eqref{e:prod} we deduce that $b(G) \leqs 6$ if $b(L) \leqs 5$ (since $|\Delta| \geqs 5$).

Therefore, we may assume $b(L)=6$, in which case Theorem \ref{t:4} implies that $L = {\rm Sp}_{6}(2)$, $K = L_{\delta} = {\rm O}_{6}^{+}(2)$ and $|\Delta|=36$. Here $L=T$ and thus $G = L \wr P$. In particular, $P=Q$ is solvable. In this case, the bound in \eqref{e:prod} gives $b(G) \leqs 7$ and further work is needed to show that $b(G) \leqs 6$. Let ${\rm reg}(L,6)$ be the number of regular orbits of $L$ with respect to the natural coordinatewise action of $L$ on $\Delta^6$. Since $d(P) \leqs 5$, \cite[Theorem 2.13]{BC} implies that $b(G) \leqs 6$ if ${\rm reg}(L,6) \geqs 5$. As in the proof of \cite[Theorem 8.2]{Bur20}, we have ${\rm reg}(L,6) \geqs 5t/|L|$, where
\[
t = |\{ (\delta,\lambda_1, \ldots, \lambda_5) \in \Delta^6 \,:\, \bigcap_i K_{\lambda_i} = 1\}|,
\]
so ${\rm reg}(L,6) \geqs 5$ if $t \geqs |L|$. This is a straightforward {\sc Magma} computation.

\vs

\noindent \emph{Case 2. $L$ is diagonal type.}

\vs

To complete the proof of the proposition, we may assume $L$ is a diagonal type group with socle $T=S^{\ell}$, where $\ell \geqs 2$ and $S$ is a nonabelian simple group. Recall that
\[
S^{\ell} \normeq L \leqs S^{\ell}.({\rm Out}(S) \times A),
\]
where $A \leqs S_{\ell}$ is the group induced by the conjugation action of $L$ on the $\ell$ factors of $T = S^{\ell}$ (note that $A$ is either primitive, or $\ell=2$ and $A=1$).

If $\ell=2$ then \cite[Theorem 1.2]{fawcett} gives $b(L) \leqs 4$ and by repeating the argument in Case 1 we deduce that $b(G) \leqs 6$ via \eqref{e:prod}. Now assume $\ell \geqs 3$. As explained in \cite[Section 4.3.2]{HLM}, there exist three elements in $\O$ whose pointwise stabilizer in $G$ is a permutation group $R \leqs {\rm Sym}(\Lambda)$ with $|\Lambda| = k\ell$ such that
\begin{equation}\label{e:R}
b(G) \leqs 2\frac{\log d(R)}{\log |S|}+4.
\end{equation}
Since $R$ is solvable, Theorem \ref{t:dist}(i) gives $d(R) \leqs 5$ and the result follows since $|S| \geqs 60$.
\end{proof}

\subsection{Twisted wreath products}\label{ss:tw}

\begin{prop}\label{p:tw2}
Let $G \leqs {\rm Sym}(\O)$ be a finite primitive group of twisted wreath product type with point stabilizer $H$. If every $2$-point stabilizer in $G$ is solvable, then $b(G) \leqs 6$.
\end{prop}

\begin{proof}
Here $G = T^k{:}H$, $H \leqs S_k$ is transitive and $n = |\O| = |T|^k$, where $T$ is a nonabelian simple group. Let $L \leqs H$ be a $2$-point stabilizer and observe that $|G:L| \leqs n^2$. Since $L \leqs S_k$ is solvable, Lemma \ref{l:2} implies that $|L|< 5^{k-1}$ and so by combining the upper bound on $b(G)$ in \eqref{e:faw} with the obvious bound $|T| \geqs 60$, we deduce that
\[
b(G) < \frac{\log |G|}{\log n}+4 < \frac{(k-1)\log 5}{k \log 60}+6 < 7
\]
and the result follows.
\end{proof}

\vs

This completes the proof of Theorem \ref{t:4}. We close this section with some further comments on twisted wreath products, as well as the affine groups excluded in Theorem \ref{t:4}.

\begin{rem}\label{r:tw2}
Let $G = T^k{:}H$ be a twisted wreath product, where $T$ is a nonabelian simple group and $H \leqs S_k$ is transitive. Recall that $H$ is nonsolvable. It is worth noting that there are primitive groups of this form with the property that every $2$-point stabilizer is solvable (we thank Michael Giudici for drawing our attention to \cite{GLPST} and the following example). Following \cite[Example 4.15]{GLPST}, we can construct a primitive twisted wreath product with $T = A_5$, $k=6$ and $H = A_6$. By \cite[Theorem 1.5]{GLPST}, the minimal subdegree of $G$ is $12$ and we conclude that every $2$-point stabilizer in $G$ is solvable (since every proper nonsolvable subgroup of $H$ has index $6$). In this particular example, \cite[Theorem 1.1]{fawcett_tw} implies that $b(G)=2$. Indeed, this result states that $b(G)=2$ whenever $H$ is a quasiprimitive subgroup of $S_k$ (there are only partial base size results when $H$ is imprimitive and it is worth noting that $b(G)$ can be arbitrarily large in general).
\end{rem}

\begin{rem}\label{r:fn}
Let $G = VH$ be an affine group. Here the condition on $2$-point stabilizers implies that $H_v$ is solvable for all nonzero vectors $v \in V$ and it is worth noting that there are genuine examples with $H$ nonsolvable and $b(G) = 5$. For example, if $G = 2^4{:}{\rm Sp}_{4}(2)$ then $G$ is $2$-transitive, every $2$-point stabilizer is isomorphic to $2^4{:}(S_4 \times S_2)$ and one checks that $b(G) = 5$. By inspecting the {\sc Magma} database of primitive groups, it is straightforward to check that $b(G) \leqs 5$ for all relevant affine groups of degree at most $4095$. In the general case, if there exists a nonzero vector $v$ such that $H_v$ acts completely reducibly on $V$, then a theorem of Halasi and Mar\'{o}ti \cite[Theorem 1.1]{HM} implies that $b(G) \leqs 5$. The case where no $2$-point stabilizer $H_v$ is completely reducible remains open.
\end{rem}

\section{Proof of Theorem \ref{t:5}}\label{s:fn}

In this final section we prove Theorem \ref{t:5}. Let $G \leqs {\rm Sym}(\O)$ be a finite primitive permutation group with point stabilizer $H$ and assume $G$ is not of affine type. Set $n = |\O|$. Fix positive integers $c,d$ with $d \geqs 5$ and assume every $c$-point stabilizer in $G$ is in $\Gamma_d$. Our aim is to establish the existence of a function $g(c,d)$, which is linear in $c$ and $d$, such that $b(G) \leqs g(c,d)$ for all primitive groups $G$ of this form.

Notice that if $L$ is a $c$-point stabilizer then $|G:L| \leqs n^c$. Since the main theorem of \cite{HLM} gives
\[
b(G) \leqs 2\frac{\log |G|}{\log n} + 24,
\]
it suffices to show that $|L| \leqs n^{h(c,d)}$ for some function $h$ (linear in $c$ and $d$).

\begin{rem}
Let $c,d$ be positive integers as above. Then there exists an almost simple primitive group $G$ such that every $c$-point stabilizer is in $\Gamma_d$ and $b(G) = c+d-2$. For example, take the natural action of $G = S_{n}$ with $n=c+d-1$.
\end{rem}

\begin{prop}\label{p:fn2}
The conclusion to Theorem \ref{t:5} holds when $G$ is almost simple.
\end{prop}

\begin{proof}
Let $G \leqs {\rm Sym}(\O)$ be an almost simple group with socle $G_0$. Recall that if $G$ is non-standard, then $b(G) \leqs 7$, so we may assume $G$ is standard.

First assume $G_0 = A_m$ and $\O$ is the set of $k$-element subsets of $\{1,\ldots, m\}$, where $1 \leqs k < m/2$. Note that if $L$ is a $c$-point stabilizer, then
\[
|L| \leqs |H| \leqs |S_{m-k}||S_k| < (m-k)!^2 \leqs (m-k)^{2(m-k)} \leqs n ^{2(m-k)}.
\]
If $c \geqs m-k$ then $|L| \leqs n^{2c}$ and the result follows. On the other hand, if $c<m-k$ then the pointwise stabilizer of the $k$-sets
$\a_i = \{1,\ldots, k-1,k+i\}$ with $i=0,1, \ldots, c-1$ clearly contains $A_{m-k-c+1}$. This implies that $m-k \leqs c+d-2$ and we deduce that $|L| \leqs n^{2(c+d-2)}$, which gives the desired result.

Next suppose $G_0 = A_m$ and $\O$ is the set of partitions of $\{1, \ldots, m\}$ into $m/k$ parts of size $k$, where $1 < k \leqs m/2$ and $k$ divides $m$. Recall that the main theorem of \cite{BCN} gives
\[
b(G) \leqs \max\left\{6, \left\lceil \log_{m/k} k \right\rceil+3\right\}
\]
and so we may assume that $c \leqs k$.  Visibly, the point stabilizer of the following partitions
\[
\a_i = \left\{ \{1, \ldots, k-1, k+i\}, \{k,k+1, \ldots, 2k\} \setminus \{k+i\}, \ldots \right\} \in \O
\]
for $i=0,1, \ldots, c-1$ contains $A_{k-1}$ (note that all of these partitions only differ in the first two parts). Therefore, $k \leqs d$ and the argument in this case is complete.

To complete the proof, we may assume $G_0$ is a classical group in a subspace action. Let $V$ be the natural module for $G_0$ and set $m = \dim V$. There are three cases to consider:
\begin{itemize}\addtolength{\itemsep}{0.2\baselineskip}
\item[{\rm (a)}] $\O$ is a set of $k$-dimensional subspaces of $V$;
\item[{\rm (b)}] $G_0 = {\rm L}_{m}(q)$ and $\O$ is a set of pairs of subspaces $\{U,W\}$ of $V$ such that $\dim U = k$, $\dim W = m-k$ and either $V = U \oplus W$ or $U \subseteq W$;
\item[{\rm (c)}] $G_0 = {\rm Sp}_{m}(q)$, $q$ is even and $H \cap G_0 = {\rm O}_{m}^{\e}(q)$.
\end{itemize}

First consider (a). By combining \cite[Theorem 3.3]{HLM} and the proof of \cite[Theorem 3.1]{HLM}, we deduce that $b(G) \leqs m/k+14$ and so we may assume $m$ is large and $c \leqs m/4k$. Then we can choose $c$ spaces in $\O$ such that the pointwise stabilizer in $G_0$ of these spaces has a section isomorphic to a simple classical group with a natural module of dimension $\lfloor m/4\rfloor$. For example, suppose $G_0 = {\rm PSp}_{m}(q)$ and $\O$ is the set of $k$-dimensional nondegenerate subspaces of $V$, so $k$ is even. Let $\{e_1, \ldots, e_{m/2}, f_1, \ldots, f_{m/2}\}$ be a standard basis for $V$ and consider the $k$-spaces
\[
U_i = \la e_{1+ik/2}, \ldots, e_{(i+1)k/2}, f_{1+ik/2}, \ldots, f_{(i+1)k/2} \ra \in \O
\]
with $i=0,1, \ldots, c-1$. Then the point stabilizer of these $k$-spaces contains a subgroup ${\rm PSp}_{m-ck}(q)$ acting trivially on the $ck$-space $U_0 \perp \cdots \perp U_{c-1}$ and the claim follows since $m-ck \geqs m/4$. A similar argument applies in the remaining cases. Finally, Lemma \ref{l:3} now implies that $\lfloor m/4\rfloor \leqs 2d$ and the result follows.

A very similar argument applies in case (b) and we omit the details. By identifying ${\rm Sp}_{m}(q)$ with the orthogonal group ${\rm O}_{m+1}(q)$, case (c) can also be handled in a similar fashion, noting that $b(G) \leqs m+4$ by the proof of \cite[Theorem 3.1]{HLM}. To explain how this plays out, we proceed as in Case 3 in the proof of Proposition \ref{p:6class}. Let $W$ be the natural module for ${\rm O}_{m+1}(q)$ with corresponding quadratic form $Q$ and let $\la v \ra$ be the radical of the associated bilinear form. Note that $v$ is nonsingular with respect to $Q$, so $Q(v) \ne 0$. Identify $\O$ with the set of nondegenerate $m$-dimensional subspaces of $W$ of type $\e$ and recall that an $m$-space is nondegenerate if and only if it does not contain $v$. Also recall that we may assume $c \leqs m/4$ and $m$ is large. Fix a basis
\[
\{e_1, \ldots, e_{m/2-1}, f_1, \ldots, f_{m/2-1},x,y,v\}
\]
for $W$, where $\la e_1, \ldots, e_{m/2-1},f_1, \ldots, f_{m/2-1}\ra$ and $\la x,y \ra$ are nondegenerate spaces of plus-type and $\e$-type, respectively. Set
\[
U_0 = \la e_1, \ldots, e_{m/2-1}, f_1, \ldots, f_{m/2-1},x,y\ra \in \O
\]
and for $i = 1, \ldots, c-1$ define
\[
U_i = \la e_1, \ldots, e_{i-1}, e_i+v, e_{i+1}, \ldots, e_{m/2-1}, f_1, \ldots, f_{m/2-1},x,y\ra \in \O.
\]
Then the pointwise stabilizer in $G_0$ of the spaces $U_0, \ldots, U_{c-1}$ visibly has a section isomorphic to $\O_{m-2c+2}^{\e}(q)$ (acting trivially on the plus-type space $\la e_1, \ldots, e_{c-1}, f_1, \ldots, f_{c-1}\ra$) and we complete the proof by arguing as above.
\end{proof}

\begin{prop}\label{p:fn1}
The conclusion to Theorem \ref{t:5} holds when $G$ is a twisted wreath product.
\end{prop}

\begin{proof}
It is easy to generalize the proof of Proposition \ref{p:tw2}. Write $G = T^k{:}H$, where $H \leqs S_k$ is transitive, $n = |\O| = |T|^k$ and $T$ is a nonabelian simple group. Let $L \leqs H$ be a $c$-point stabilizer and observe that $|G:L| \leqs n^c$. Since $L \leqs S_k$, Lemma \ref{l:2} implies that $|L|< d^{k-1}$ and thus \eqref{e:faw} yields
\[
b(G) < \frac{\log |G|}{\log n}+4 < \frac{(k-1)\log d}{k \log |T|}+c+4 < \frac{\log d}{\log 60}+c+4
\]
and the result follows.
\end{proof}

\begin{prop}\label{p:fn5}
The conclusion to Theorem \ref{t:5} holds when $G$ is of diagonal type.
\end{prop}

\begin{proof}
We adopt the notation from the proof of Proposition \ref{p:fn3}, so $G$ has socle $N=T^k$ and we have $N \leqs G \leqs N.({\rm Out}(T) \times S_k)$. By arguing as in the proof of Proposition \ref{p:fn3} we may assume that $k$ is large, $c \leqs k/4$ and $N{:}A_k \leqs G$.
Fix $1 \ne s \in T$ and define $g_i = (1, \ldots, 1, s, 1, \ldots, 1) \in N$ for $i = 1, \ldots, c-1$, where $s$ is the $i$-th coordinate of $g_i$. Then $A_{k-c+1}$ is contained in the $c$-point stabilizer $H \cap H^{g_1} \cap \cdots \cap H^{g_{c-1}}$ and thus $k \leqs c+d-2$. The result now follows by applying the upper bound on $b(G)$ in \eqref{e:diag}.
\end{proof}

\begin{prop}\label{p:fn6}
The conclusion to Theorem \ref{t:5} holds when $G$ is of product type.
\end{prop}

\begin{proof}
We proceed as in the proof of Proposition \ref{p:fn4} and we freely adopt the notation therein,  which is consistent with the notation from the proof of Proposition \ref{p:prod}. In particular, we have $N \leqs G \leqs L \wr P$, where $N = T^k$ is the socle of $G$ and $L \leqs {\rm Sym}(\Delta)$ is a primitive group with socle $T$.

\vs

\noindent \emph{Case 1. $L$ is almost simple.}

\vs

First assume $L$ is almost simple and $c>|\Delta|$. Set $e=c$ and fix a base $\{\delta_1, \ldots, \delta_b\}$ for $L$, where $b=b(L)$. Choose $c$ distinct points $\a_1, \ldots, \a_c$ in $\O$ such that $\a_i = (\delta_i, \ldots, \delta_i)$ for $i = 1, \ldots, b$. Then
\[
R:=\bigcap_{i=1}^{c}G_{\a_i} \leqs \{(x_1, \ldots, x_k)\pi \in G \,:\, \mbox{$x_j=1$ for all $j$}\}
\]
is in $\Gamma_d$ and thus $d(R) \leqs d$
by Theorem \ref{t:dist}(ii). Finally, by applying \cite[Lemma 3.8]{BS} and \cite[Lemma 2.1]{DHM} we see that
\begin{equation}\label{e:prod2}
b(G) \leqs \left\lceil \log_{|\Delta|} d(R) \right\rceil+e
\end{equation}
and the result follows.

Next assume $L$ is almost simple and $c \leqs |\Delta|$. Choose any $c$ distinct points $\delta_1, \ldots, \delta_c$ in $\Delta$ and set $\a_i = (\delta_i, \ldots, \delta_i) \in \O$. Since the pointwise stabilizer in $G$ of these points is in $\Gamma_d$, it follows that the pointwise stabilizer of the $\delta_i$ in $T$ is also in $\Gamma_d$, which in turn implies that $\bigcap_i L_{\delta_i} \in \Gamma_d$ since $L/T$ is solvable. Therefore, every $c$-point stabilizer in $L$ is in $\Gamma_d$, so Proposition \ref{p:fn2} implies that $b(L) \leqs h(c,d)$ for some function $h$, which is linear in $c$ and $d$.

Let $e=\max\{c,b(L)\}$ and choose $e$ distinct points $\{\delta_1, \ldots, \delta_e\} \subseteq \Delta$ containing a base for $L$. Let $R$ be the pointwise stabilizer in $G$ of the elements $\a_i = (\delta_i, \ldots, \delta_i) \in \O$. Then $R \in \Gamma_d$ and by arguing as above we deduce that \eqref{e:prod2} holds. The desired result follows via the bound on $d(R)$ in Theorem \ref{t:dist}(ii) and the fact that $e \leqs h(c,d)+c$.

\vs

\noindent \emph{Case 2. $L$ is diagonal type.}

\vs

Now assume $L$ is a diagonal type group with socle $T=S^{\ell}$, where $\ell \geqs 2$ and $S$ is a nonabelian simple group, so
\[
S^{\ell} \normeq L \leqs S^{\ell}.({\rm Out}(S) \times A)
\]
and $A \leqs S_{\ell}$ is the group induced by the conjugation action of $L$ on the $\ell$ factors of $T$. If $\ell=2$ then $b(L) \leqs 4$ by \cite[Theorem 1.2]{fawcett} and we can repeat the argument in Case 1. For the remainder, we will assume $\ell \geqs 3$. Set $e=\max\{3,c\}$. By arguing as in \cite[Section 4.3.2]{HLM}, we may choose $e$ distinct points in $\O$ with the property that their pointwise stabilizer in $G$ is a permutation group $R \leqs {\rm Sym}(\Lambda)$ with $|\Lambda|=k\ell$ such that
\[
b(G) \leqs 2\frac{\log d(R)}{\log |S|}+e+1
\]
(see \eqref{e:R}). Since $R \in \Gamma_d$ we have $d(R) \leqs d$ by Theorem \ref{t:dist}(ii) and the result follows.
\end{proof}

\vs

This completes the proof of Theorem \ref{t:5}.

\end{document}